\documentclass[reqno]{amsart}
\usepackage{amsmath, amssymb, amsthm}
\usepackage{url}
\usepackage[hidelinks]{hyperref}
\usepackage{mathtools}
\usepackage{upgreek}
\usepackage{environ}
\newcommand{\cosec}{\text{cosec}}

\numberwithin{equation}{section}
\theoremstyle{plain}
\newtheorem{theorem}{Theorem}[section]
\newtheorem{lemma}[theorem]{Lemma}

\newtheorem{proposition}[theorem]{Proposition}

\newtheorem*{definition*}{Definition}

\begin{document}
	\title[The Piltz divisor Problem in Number Field
	]{The Piltz divisor Problem in Number Fields Using The Resonance Method} 
	
	\author{Nilmoni Karak}
	\address{Nilmoni Karak\\ Department of Mathematics \\
		Indian Institute of Technology Kharagpur\\
		Kharagpur-721302,  India.}
	\email{nilmonikarak@gmail.com,  nilmonimath@kgpian.iitkgp.ac.in}
	
	\author{Kamalakshya Mahatab}
	\address{Kamalaksya Mahatab\\ Department of Mathematics \\
		Indian Institute of Technology Kharagpur\\
		Kharagpur-721302,  India.}
	\email{  kamalakshya@maths.iitkgp.ac.in}
	

	\begin{abstract}
The Piltz divisor problem is a natural generalization of the classical Dirichlet divisor problem. In this paper, we study this problem over number fields and obtain improved 
$\Omega-$bounds for its error terms. Our approach involves generalizing a Voronoi-type formula due to Soundararajan in the number field setting, and applying a recent result due to the second author. 
	\end{abstract}
	\maketitle
	
	\section{Introduction}
	The study of error terms in various counting functions is a central theme in analytic number theory. The Dirichlet divisor problem and the Gauss circle problem are two classical examples that ask to estimate the discrepancy between an arithmetic counting function and its asymptotic approximation. 
	The error terms of the above said problems are 
	\begin{equation}\label{Circle and Divisor Problem}
\Delta(x) = \sum_{n\leq x} d(n) -x\log x -(2\gamma -1)x,	 \ \ \ \text{and} \ \ \	R(x) = \sum_{\substack{{a^2+b^2\leq x}\\ a,b \, \in \mathbb{Z}}}1 - \pi x,
	\end{equation}
respectively, where $d(n)$ counts the number of divisors for a positive integer $n$ and $\gamma=0.5772\cdots$ is the Euler's constant. Both, the Gauss circle problem and  the Dirichlet divisor problem seek upper bounds of the form\footnote{The symbol $\epsilon$ always denotes an arbitrary small positive constant, whose value may varry as needed.}  $R(x)\ll x^{\theta_1+\epsilon}$ and $\Delta(x)\ll x^{\theta_2 +\epsilon}$, with $\theta_i$  as small as possible. It is conjectured that $\theta_i =1/4$ for both the problems. The sharpest known value of $\theta_i$, due to Huxley \cite{Huxley 2, Huxley 3}, has long been $\theta_i = \frac{131}{416} = 0.3149\cdots$. However, Li and Yang \cite{Li- Yang } recently refined this result by improving the bound at the fourth decimal place, showing that $\theta = 0.314483\cdots$.
 It is also of interest to understand how large $R(x)$ and $\Delta(x)$ could be as $x \to \infty$. In this direction, Soundararajan \cite{Sound} established \footnote{Throughout this paper we denote  $\log_2 X = \log \log X$ and $\log_3 X= \log \log \log X$.}  
	\begin{equation*}
		\Delta(x) = \Omega\left((x\log x)^{1/4} (\log_2 x)^{\frac{3}{4}(2^{\frac{4}{3}-1})}(\log_3 x)^{-5/8}\right)
	\end{equation*}
	and 
		\begin{equation*}
		R(x) = \Omega\left((x\log x)^{1/4} (\log_2 x)^{\frac{3}{4}(2^{\frac{4}{3}-1})}(\log_3 x)^{-5/8}\right).
	\end{equation*}
	Using the resonance method, the second author \cite{K. Mahatab} has improved these $\Omega$-bounds by raising the exponent of $\log_3 x$  from $-5/8$ to $-3/8$. The aim of this article is to extend the approach in \cite{K. Mahatab} in number fields, which includes the Gauss circle problem and the Dirichlet divisor problem.

	Let $\mathbb{K}$ be an algebraic number field of degree $[\mathbb{K}: \mathbb{Q}] =m=r_1+2r_2$, where $r_1$ and $r_2$ denote the number of real and complex embeddings of $\mathbb{K}$, respectively, and $\mathcal{O}_{\mathbb{K}}$ be its ring of integers. Let $D$ be the absolute value of the discriminant of $\mathbb{K}$ and $N(\mathfrak{a})$ denotes the absolute norm of $\mathfrak{a}$. Then the Dedekind zeta function associated to $\mathbb{K}$ is defined as 
	\begin{equation*}
		\zeta_{\mathbb{K}}(s)= \sum_{\mathfrak{a}\subset \mathcal{O}_{\mathbb{K}}} \frac{1}{N(\mathfrak{a})^s} = \prod_{\mathfrak{p}\subset \mathcal{O}_{\mathbb{K}}} \left(1-\frac{1}{N(\mathfrak{p})^s}\right)^{-1} ,
	\end{equation*} 
for all $s\in \mathbb{C}$ with $\text{Re}(s)>1$, 
	where the sum and the product are taken over all non zero integral ideals $\mathfrak{a}$ and prime ideals $\mathfrak{p}$ of  $\mathcal{O}_{\mathbb{K}}$, respectively. The function $\zeta_{\mathbb{K}}(s)$ can be continued analytically to a meromorphic function with a unique simple pole at $s=1$. Moreover, it satisfies the functional equation $\zeta_{\mathbb{K}}(s)= \chi_{\mathbb{K}}(s) \zeta_{\mathbb{K}}(1-s)$, where
	\begin{equation} \label{chi_K(s)}
		\chi_{\mathbb{K}}(s) = 2^{m(s-1)} \pi^{ms} D^{1/2-s} \cfrac{\sec^{r_1+r_2}(\pi s/2)\,\cosec^{r_2}(\pi s/2)}{\Gamma(s)^m}.
	\end{equation}

The number field analogue of $k$-fold divisor function is defined by 
\begin{equation}
	d_{\mathbb{K}}^{(k)}(n) = \sum_{N(\mathfrak{a}_1 \mathfrak{a}_2 \cdots \mathfrak{a}_k) = n} 1,
\end{equation}
where the sum is taken over $k $-tuples of ideals in the number field $ \mathbb{K}$ whose product has norm $n$. The arithmetic function
$d_{\mathbb{K}}^{(k)}(n)$ satisfies the following identity:
	\begin{equation*}
		\zeta_{\mathbb{K}}^{k}(s) =\sum_{n=1}^{\infty} \frac{d_{\mathbb{K}}^{(k)}(n)}{n^s},  \ \ \ \ \ \ \  \ \text{Re}(s)>1.
	\end{equation*} 
	In analogy to the classical Dirichlet divisor problem in \eqref{Circle and Divisor Problem}, our aim is to investigate the behavior of the error term in the asymptotic formula
	\begin{equation}\label{Asymp of sum d_K}
		\Delta_{\mathbb{K}}^{(k)}(x)=\sum_{n\leq x} d_{\mathbb{K}}^{(k)}(n) - \underset{w=1}{\operatorname{Res}} \left(\zeta_{\mathbb{K}}^k(w)\frac{x^w}{w}\right),
	\end{equation}
	as $x\rightarrow \infty$. The problem of evaluating the error term $\Delta_{\mathbb{Q}}^{(k)}$ is called the \emph{Piltz divisor problem}.  There are very few results on this problem over number field. Nowak \cite{Nowak_1993} showed that the upper bound estimate for $\Delta_{\mathbb{K}}^{(k)}$ is as follows:
\begin{equation*}
	\Delta_{\mathbb{K}}^{(k)}(x) \ll
	\begin{cases}
		x^{1 - \frac{2}{mk} + \frac{8}{mk(5mk + 2)}} (\log x)^{k - 1 - \frac{10(k - 2)}{5m + 2}},
		& \text{for } 3 \leq mk \leq 6, \\
		x^{1 - \frac{2}{mk} + \frac{3}{2m^2k^2}} (\log x)^{k - 1 - \frac{2(k - 2)}{mk}},
		& \text{for } mk \geq 7,
	\end{cases}
\end{equation*}
where $m=[\mathbb{K}: \mathbb{Q}] \geq 2$. Recently, Takeda \cite{Takeda} improved this upper bound, showing  that 
\begin{equation*}
	\Delta_{\mathbb{K}}^{(k)}(x) \ll_{\epsilon} x^{(2mk-3)/(2mk+1)+\epsilon} D^{2k/(2mk+1)+\epsilon},
\end{equation*} holds for every $mk\geq 4$. 

For the lower bound estimate, we need to introduce some notations from algebraic number theory. Most of the $\Omega$-estimates for $ \Delta_{\mathbb{K}}^{(k)}$ depend on Dirichlet densities.
 For $0\leq \nu \leq m$, $P_{\nu}$ denotes the set of all rational primes which are unramified in $\mathbb{K}$  and which split into exactly $\nu$ distinct $\mathcal{O}_{\mathbb{K}}$-prime ideals of degree $1$. To define the Dirichlet density $\delta_{\nu}$ precisely, we set 
$\mathbb{K}^{G}$ to be the Galois closure of $\mathbb{K}/\mathbb{Q}$, $G=\text{Gal}\left(\mathbb{K}^G/\mathbb{Q}\right)$ its Galois group, and $H= \text{Gal}\left(\mathbb{K}^G/\mathbb{K}\right)$ the subgroup of $G$ corresponding to $\mathbb{K}$. Then the Dirichlet densities $\delta_{\nu}$ are defined as
\begin{equation} \label{Def of Density}
	\delta_{\nu}= \cfrac{\left| \left\{ \sigma\in G: \left| \left\{\psi \in G: \sigma \in \psi H\psi^{-1}\right\}\right| = \nu \left|H\right|\right\} \right|}{\left|G\right| },
\end{equation}
	and these constants satisfy the condition $\sum_{\nu=1}^m \nu \delta_{\nu}=1$. Let $R$ be the cardinality of the set $\mathcal{I} = \{1\leq \nu \leq m: \delta_{\nu}>0\}$. 
	In this setup, Szeg\"o{} and Walfisz \cite{Szego-Walfisz, Szego-Walfisz (2)} first derived the estimate 
	\begin{equation*}
		\Delta_{\mathbb{K}}^{(k)}(x) = \Omega_{*}\left((x\log x)^{\frac{mk-1}{2mk}}(\log_2 x)^{k-1}\right).
	\end{equation*}
	Later, Hafner \cite{Halfner(2)} successfully improved this lower bound and proved 
	\begin{equation*}
		\Delta_{\mathbb{K}}^{(k)}(x) = \Omega_{*}\left( (x\log x)^{\frac{mk-1}{2mk}} (\log_2 x)^{\beta^{'}} \exp(-c_3\sqrt{\log_3x})\right),
	\end{equation*}
	where $c_3$ is some positive constant and 
	\begin{equation*}
		\beta' = \frac{mk-1}{2mk} \left( \left(\sum_{\nu=1}^{k}\delta_{\nu}\, \nu \log \nu\right)k +k\log k -k +1\right)+k -1.
	\end{equation*}
	Here, the $\Omega$-estimates are defined as follows:
\begin{align*}
	\Omega_{*} = 
	\begin{cases}
		\Omega_{\pm}, & \text{if } mk \geq 4 \text{ or } \mathbb{K} \text{ is cubic and not totally real}, \\
		\Omega_{+}, & \text{if } k = 2,3 \text{ and } \mathbb{K} = \mathbb{Q}, \text{ or } k = 1, m = 2,3 \text{ and } \mathbb{K} \text{ is totally real},\\
		\Omega_{-}, & \text{if } k = 1 \text{ and } \mathbb{K} \text{ is quadratic imaginary}.
	\end{cases}
\end{align*}
Using Soundararajan's \cite{Sound} method, Girstmair et al.\ \cite{GKMN} refined Hafner's \cite{Halfner(2)} $\Omega$-result.
 Indeed, they showed that 
\begin{equation}\label{Omega bound of Girstmair et al}
	\Delta_{\mathbb{K}}^{(k)}(x) = \Omega\Big( (x\log x)^{\frac{mk-1}{2mk}} (\log_2x)^{\beta} (\log_3 x)^{\gamma'} \Big),
\end{equation}
	with
\begin{equation}\label{beta and gamma}
	\beta = \frac{mk+1}{2mk} \left( \sum_{\nu =1}^{m}\delta_{\nu} (k\nu )^{\frac{2mk}{mk+1}}-1\right), \ \  \ \gamma'= - \frac{mk+1}{4mk}{R} - \frac{mk-1}{2mk}.
\end{equation}
Motivated by the approach in \cite{AMM}, the second author uses resonance method to general exponential sums and gives detailed lower bounds for these sums in \cite[Theorem 5]{K. Mahatab}. We use this to improve the $\Omega$-bounds in \eqref{Omega bound of Girstmair et al}, which yields our main result:
	\begin{theorem} \label{Main Result}
	Let $\mathbb{K}$ be any arbitrary algebraic number field of degree $m$ and $k$ be any positive integer with $mk\geq 2$. Then, for $x \rightarrow \infty$, the error term $\Delta_{\mathbb{K}}^{(k)}(x)$ in \eqref{Asymp of sum d_K} satisfies 
	\begin{equation} \label{final result}
		\Delta_{\mathbb{K}}^{(k)}(x) = \Omega\Big( (x\log x)^{\frac{km-1}{2km}} (\log_2 x)^{\beta} (\log_3 x)^{\gamma} \Big),
	\end{equation} 
	where 
	\begin{equation*}
	\beta = \frac{mk+1}{2mk} \left( \sum_{\nu =1}^{m}\delta_{\nu} (k\nu )^{\frac{2mk}{mk+1}}-1\right) \hspace{0.5cm}\text{and} \hspace{0.5cm}\gamma= - \frac{km+1}{4mk}{R}.
	\end{equation*}
 Moreover, if $r_1$ denotes the number of real conjugates of  $\mathbb{K}$,
 in \eqref{final result} we may replace $\Omega$ with $\Omega_{+}$  when $kr_1 \equiv 3 \,(\mathrm{mod} 8)$ and $\Omega$ with $\Omega_{-}$ when $kr_1 \equiv 7 \,(\mathrm{mod} 8)$.
	\end{theorem}
	The above theorem includes \eqref{Circle and Divisor Problem} as special cases. The Gauss circle problem corresponds to $\mathbb{K}=\mathbb{Q}(i)$ and $k=1$, while the Dirichlet divisor problem corresponds to $\mathbb{K}=\mathbb{Q}$ and $k=2$. In Section 4, we discuss some special examples.

The structure of the paper is as follows. In Section 2, we present a result due to the second author, which serves as a key ingredient in the proof of our main theorem. We also generalize a Voronoi-type formula due to Soundararajan 
\cite{Sound} in the context of number fields. Section 3 is devoted to the proof of Theorem \ref{Main Result}, and in Section 4, we discuss some immediate examples of this theorem.

	\section{Preliminaries}
		 We recall a theorem due to the second author \cite{K. Mahatab}, which plays a central role in establishing our main result. Let $(\lambda_n)_{n=1}^{\infty}$ be a non-decreasing sequence of positive real numbers  and $\alpha$ be a positive real parameter. We consider a finite set $\mathcal{M}\subseteq \{\lambda_n: C_1 \alpha \leq \lambda_n\leq 2\alpha\}$, where $0<C_1<2$ such that the elements of $\mathcal{M}$ are linearly independent over $\mathbb{Q}$. Let $M$ be the cardinality of the set $\mathcal{M}$. Additionally, given $A_1$, we consider real numbers $A_2, A_3, A_4$ satisfying $0 < A_4 < A_3 < A_2 < A_1$, and let $\theta$ be any real number.
		 
	\begin{theorem} \label{K-M Theorem}
		Let $(f(n))_{n=1}^{\infty}$ be a sequence of positive real numbers. Then, for large $X$, we have
		\begin{align*}
			\max_{X^{A_3}/2\leq x\leq 2A_2^2X^{A_2}\log^2 X}\left| \sum_{n\leq X^{A_1}} f(n) \cos\big(x\lambda_n+\theta\big)\right|\geq \frac{\pi}{4e}\sum_{n\in \mathcal{M}}f(n) &+ O\left( X^{A_3-A_2}e^{2M/C_1}\left(\sum_{\lambda_n\leq 4\alpha}f(n)\right)\right) \\
			& + O\left(\frac{X^{-A_4}}{\alpha}\sum_{n\leq X^{A_1}}f(n)\right).
		\end{align*}
	\end{theorem}
To complete the proof of our main  theorem, we prove the following proposition, which is a generalization of Soundararajan's result \cite{Sound} in number fields. This proposition contains a Vornoi type summation formula for the error term $\Delta_{\mathbb{K}}^{(k)}(x)$.
\begin{proposition} \label{Prop}
	Let $\mathbb{K}$ be any number field of degree $m=r_1+ 2r_2$, where $r_1$ and $r_2$ denote the number of real and complex embeddings of $\mathbb{K}$, respectively. Let $D$ be the absolute value of the discriminant of $\mathbb{K}$. Suppose $x\geq2$ and $\alpha\geq 2$ be two real parameters. Then, for a fixed positive integer $k$, we have
	\begin{align} \label{Proposition equation}
		\nonumber&\frac{\alpha^{\frac{1}{mk}}}{\sqrt{\pi}}\int_{-\infty}^{\infty}\Delta_{\mathbb{K}}^{(k)}\left(x^{mk}e^{u/x}\right) \exp\left(-u^2\alpha^{\frac{2}{mk}}\right)\mathrm{d}u = O\left( t^{\frac{mk}{2}-\frac{3}{5}}\alpha^{\frac{2}{mk}}\right)\\
		& \ \ \ +\frac{D^{\frac{1}{2m}}}{\pi\sqrt{mk}} x^{\frac{mk-1}{2}}\sum_{n=1}^{\infty}\frac{d_\mathbb{K}^{(k)}(n)}{n^{\frac{mk+1}{2mk}}} \exp\left(-\pi^2\left(\frac{n}{D^k \alpha}\right)^{\frac{2}{mk}}\right) \cos\left( 2\pi mk\left(\frac{n}{D^k}\right)^{\frac{1}{mk}}x + \frac{kr_1-3}{4}\pi\right).
	\end{align}
\end{proposition}
\begin{proof}
Let $D_\mathbb{K}^{(k)}(x)= \sum_{n\leq x}d_\mathbb{K}^{(k)}(n)$. We start with the following integral:
\begin{equation*}
	\frac{\alpha^{\frac{1}{mk}}}{\sqrt{\pi}}\int_{-\infty}^{\infty}D_{\mathbb{K}}^{(k)}\left(x^{mk}e^{u/x}\right) \exp\left(-u^2\alpha^{\frac{2}{mk}}\right)\mathrm{d}u.
\end{equation*}
Using Perron's formula, for $c>1$, the above integral becomes
\begin{align*}
	&\frac{\alpha^{\frac{1}{mk}}}{\sqrt{\pi}} \int_{-\infty}^{\infty} \int_{c-i\infty}^{c+\infty} \zeta_{\mathbb{K}}(s)^k x^{mks}\exp\left(\frac{us}{x}-u^2\alpha^{\frac{2}{mk}}\right)\frac{\mathrm{d}s}{s}\mathrm{d}u\\
	=& \frac{1}{2\pi i} \int_{c-i\infty}^{c+\infty} \zeta_{\mathbb{K}}(s)^k x^{mks}\frac{\alpha^{\frac{1}{mk}}}{\sqrt{\pi}}\int_{-\infty}^{\infty} \exp\left( - \left(u\alpha^{\frac{1}{mk}}-\frac{s}{2\alpha^{\frac{1}{mk}}x}\right)^2+\frac{s^2}{4\alpha^{\frac{2}{mk}}x^2}\right)\mathrm{d}u\frac{\mathrm{d}s}{s}\\
	=& \frac{1}{2\pi i} \int_{c-i\infty}^{c+\infty} \zeta_{\mathbb{K}}(s)^k x^{mks} \exp\left( \frac{s^2}{4\alpha^{\frac{2}{mk}}x^2}\right) \frac{\mathrm{d}s}{s}.
\end{align*}
Now we consider $d=-1/\log x$. We shift the above line of integration to the line $d-i\infty$ to $d+i\infty$. Here, we see that the integrand has poles at $s=0$ and $s=1$. Since,
\begin{align*}
	\underset{s=0}{\operatorname{Res}} \ \zeta_{\mathbb{K}}(s)^k \frac{x^{mks}}{s} \exp\left( \frac{s^2}{4\alpha^{\frac{2}{mk}}x^2}\right)= \zeta_{\mathbb{K}}(0)^k,
\end{align*}
we get an amount $O(1)$ for the pole at $s=0$ and for the pole at $s=1$, we can write the following:
\begin{align*}
		\underset{s=1}{\operatorname{Res}} \ \zeta_{\mathbb{K}}(s)^k \frac{x^{mks}}{s} \exp\left( \frac{s^2}{4\alpha^{\frac{2}{mk}}x^2}\right) = \frac{\alpha^{\frac{1}{mk}}}{\sqrt{\pi}} \int_{-\infty}^{\infty} e^{-u^2\alpha^{\frac{2}{mk}}} \left(\underset{s=0}{\operatorname{Res}} \ \zeta_{\mathbb{K}}(s)^k \frac{\left(x^{mk}e^{u/x}\right)^s}{s} \right)\mathrm{d}u.
\end{align*}
Note that for any $x>0$,
\begin{equation*}
	\Delta_{\mathbb{K}}^{(k)}(x) = D_\mathbb{K}^{(k)}(x) - \underset{s=1}{\operatorname{Res}} \ \zeta_{\mathbb{K}}(s)^k \frac{x^s}{s}.
\end{equation*}
By Cauchy residue theorem and the above identity, we arrive
\begin{align} \label{identity after using cauchy residue}
	\frac{\alpha^{\frac{1}{mk}}}{\sqrt{\pi}}\int_{-\infty}^{\infty}&\Delta_{\mathbb{K}}^{(k)}\left(x^{mk}e^{u/s}\right) \exp\left(-u^2\alpha^{\frac{2}{mk}}\right)\mathrm{d}u = \frac{1}{2\pi i} \int_{d-i\infty}^{d+i\infty} \zeta_{\mathbb{K}}(s)^k x^{mks} \exp\left( \frac{s^2}{4\alpha^{\frac{2}{mk}}x^2}\right) \frac{\mathrm{d}s}{s} +O(1).
\end{align}
Since, $\zeta_{\mathbb{K}}(s)$ satisfies the functional equation $\zeta_{\mathbb{K}}(s)= \chi_{\mathbb{K}}(s) \zeta_{\mathbb{K}}(1-s)$, where $\chi_{\mathbb{K}}(s)$ is defined in \eqref{chi_K(s)} and $\zeta_{\mathbb{K}}(1-s) =\sum_{n=1}^{\infty} d_\mathbb{K}^{(k)}(n)n^{s-1}$ holds,  equation \eqref{identity after using cauchy residue} becomes
\begin{equation} \label{Infinite sum times I_n}
\frac{\alpha^{\frac{1}{mk}}}{\sqrt{\pi}}\int_{-\infty}^{\infty}\Delta_{\mathbb{K}}^{(k)}\left(x^{mk}e^{u/s}\right) \exp\left(-u^2\alpha^{\frac{2}{mk}}\right)\mathrm{d}u=	\sum_{n=1}^{\infty} \frac{d_{\mathbb{K}}^{(k)}(n)}{n} \mathcal{I}_n +O(1),
\end{equation}
where
\begin{equation*}
	\mathcal{I}_n = \frac{1}{2\pi i} \int_{d-i\infty}^{d+i \infty} \chi_{\mathbb{K}}(s)^k \left(x^{mk}n\right)^s \exp\left(\frac{s^2}{4\alpha^{\frac{2}{mk}}x^2}\right)\frac{\mathrm{d}s}{s}.
\end{equation*}
We see that 
\begin{equation*}
	\frac{1}{2\pi i}\int_{d-i}^{d+i} \chi_{\mathbb{K}}(s)^k \left(x^{mk}n\right)^s \exp\left(\frac{s^2}{4\alpha^{\frac{2}{mk}}x^2}\right)\frac{\mathrm{d}s}{s} \ll n^d \log x.
\end{equation*}
Also, observe that the integrand of $\mathcal{I}_n$ at $d+it$ is the complex conjugate of of the integrand at $d-it$. Then
\begin{equation*}
	\mathcal{I}_n = \text{Re} \ \frac{1}{\pi i} \int_{1}^{\infty} \chi_{\mathbb{K}}(d+it)^k \left(x^{mk} n\right)^{d+it} \exp\left(\frac{d^2+2dit-t^2}{4\alpha^{\frac{2}{mk}}x^2}\right) \frac{\mathrm{d}t}{t} + O\left(n^d \log x\right).
\end{equation*}
Applying the Stirling's approximation of $\Gamma(a+it)$, one has 
\begin{equation*}
	\chi_{\mathbb{K}}(d+it) = \left(\frac{2\pi}{D^{1/m}t}\right)^{m\left(d+it-\frac{1}{2}\right)} e^{i\left(mt+\frac{\pi}{4}r_1\right)}\left(1+O\left(\frac{1}{t}\right)\right).
\end{equation*}
Thus
\begin{align} \label{Integral I_n}
	\mathcal{I}_n 
	&= \text{Re} \ \frac{1}{\pi i} \int_{1}^{\infty} 
	\left(x^{mk}n\right)^{d+it} 
	\left(\frac{2\pi}{D^{1/m}t}\right)^{mk\left(d+it-\frac{1}{2}\right)} 
	\exp\left(-\frac{t^2}{4\alpha^{\frac{2}{mk}}x^2} 
	+ ik\left(mt + \frac{\pi}{4}r_1\right)\right) 
	\frac{\mathrm{d}t}{t} + O\left( n^d x^{\frac{mk}{2}-1+\epsilon}\alpha^{1/2}\right)\\[1ex]
\nonumber	&= \text{Re} \ \frac{1}{\pi i} \int_{1}^{\infty} 
	\left(x^{mk}n\right)^d 
	\left(\frac{2\pi}{D^{1/m}t}\right)^{mk\left(d - \frac{1}{2}\right)} 
	\exp\left(-\frac{t^2}{4\alpha^{\frac{2}{mk}}x^2}\right) \\
\nonumber	&\quad \quad \quad \quad \quad \quad  \times \exp\left( i\left(
	t\log\left(x^{mk}n\right) 
	+ mkt\log\left(\frac{2\pi}{D^{1/m}t}\right) 
	+ ik\left(mt + \frac{\pi}{4}r_1\right)
	\right) \right)
	\frac{\mathrm{d}t}{t}+ O\left( n^d x^{\frac{mk}{2}-1+\epsilon}\alpha^{1/2}\right).
\end{align}
Now to compute the above integral, we use the method of stationary phase. The stationary point occurs at $t= 2\pi x \left(n/D^k\right)^{\frac{1}{mk}}$. Let $y=t-2\pi x \left(n/D^k\right)^{\frac{1}{mk}}$ and $T= \left(x\left(n/D^k\right)^{\frac{1}{mk}}\right)^{3/5}$. We divide the integral into two parts when $|y|\leq T$ and when $|y|> T$.  In the first case, using a Taylor expansion, the above integral gives
\begin{align} \label{Integral after applying Stationary}
\nonumber	\text{Re} \ 
	\frac{D^{1/m}}{2\pi^2 i} \left(xn^{\frac{1}{mk}}\right)^{\frac{mk}{2}-1} 
	\int_{-T}^{T} 
	\exp\left( 
	-\pi^2 \left(\frac{n}{D^k \alpha}\right)^{\frac{2}{mk}} 
	+ i\left( 
	2\pi mkx \left(\frac{n}{D^k}\right)^{\frac{1}{mk}} 
	+ \frac{kr_1\pi}{4} 
	- \frac{mk}{4\pi k} \left(\frac{D^k}{n}\right)^{\frac{1}{mk}} y^2 
	\right) 
	\right) \\
	\times 
	\left( 
	1 + O\left( 
	\frac{1}{\left(x\left(n/D^k\right)^{\frac{1}{mk}}\right)^{1/5}}
	\right) 
	\right)\mathrm{d}y.\ \ \ \ \ \ \ \ \ \ \ \ \ \ \  \ \ \ \ \  \ \ \ \
\end{align}
It is well known that $\int_{-T}^{T} e^{-iz^2} \mathrm{d}z = e^{-\frac{i\pi}{4}} + O\left(1/T\right)$. Using this, \eqref{Integral after applying Stationary} becomes
\begin{align}\label{1st case}
\frac{1}{\pi} \frac{D^{\frac{1}{2m}}}{\sqrt{mk}} \left(xn^{\frac{1}{mk}}\right)^{\frac{mk}{2}-1} \exp\left( -\pi^2\left(\frac{n}{D^k\alpha}\right)^{\frac{2}{mk}}\right) \left(\cos\left(2\pi mk x \left(\frac{n}{D^k}\right)^{\frac{1}{mk}}+ \frac{\pi}{4}(kr_1-3)\right)
	+  O\left(\left(x(n/D^k)^{\frac{1}{mk}}\right)^{-\frac{1}{10}}\right)\right).
\end{align}
Utilizing the exponential integral formula in (Chapter~4, \cite{Titchmarsh}), for $1\leq y\leq  2\pi x \left(n/D^k\right)^{\frac{1}{mk}} -T$, 
\begin{equation*}
	\int_{1}^{y}\exp\left(i\left(t\log\left(x^{mk}n\right)+mkt\log\left(\frac{2\pi}{D^{1/m}t}\right)+mk\left(t+\frac{3\pi}{4}\right)\right)\right)\mathrm{d}t \ll \left(\log\left(\frac{1}{t}2\pi x\left(\frac{n}{D^k}\right)^{\frac{1}{mk}}\right)\right)^{-1}.
\end{equation*}
For the second case, we consider the range $|y|>T$. Using the upper bound established above, the contribution to the integral in \eqref{Integral I_n} over the interval $1\leq t\leq 2\pi x \left(n/D^k\right)^{\frac{1}{mk}} -T$ is
\begin{equation*}
	\ll \left(xn^{\frac{1}{mk}}\right)^{\frac{mk}{2}-\frac{3}{5}}\exp\left(-\left(\frac{n}{D^k\alpha}\right)^{\frac{2}{mk}}\right) + n^d x^{\frac{mk}{2}-1}\sqrt{\alpha}.
\end{equation*}
A similar bound also holds for the interval $t\geq 2\pi x \left(n/D^k\right)^{\frac{1}{mk}} +T$. Combining all these estimates together we obtain
\begin{align*}
\mathcal{I}_n =	\frac{1}{\pi} \frac{D^{\frac{1}{2m}}}{\sqrt{mk}} \left(xn^{\frac{1}{mk}}\right)^{\frac{mk}{2}-1}& \exp\left( -\pi^2\left(\frac{n}{D^k\alpha}\right)^{\frac{2}{mk}}\right) \left(\cos\left(2\pi mk x \left(\frac{n}{D^k}\right)^{\frac{1}{mk}}+ \frac{\pi}{4}(kr_1-3)\right)\right)\\
& + O\left( n^d x^{\frac{mk}{2}-1+\epsilon}\sqrt{\alpha} + \left(xn^{\frac{1}{mk}}\right)^{\frac{mk}{2}-\frac{3}{5}}\exp\left(-\left(\frac{n}{D^k\alpha}\right)^{\frac{2}{mk}}\right)\right).
\end{align*}
Hence, the expression in \eqref{Infinite sum times I_n} becomes 
\begin{align*}
	\frac{D^{\frac{1}{2m}}}{\pi \sqrt{mk}}x^{\frac{mk-1}{2}} \sum_{n=1}^{\infty} \frac{d_{\mathbb{K}}^{(k)}(n)}{n^{\frac{mk+1}{2km}}}\exp\left( -\pi^2\left(\frac{n}{D^k\alpha}\right)^{\frac{2}{mk}}\right)\cos\left(2\pi mk x \left(\frac{n}{D^k}\right)^{\frac{1}{mk}}+ \frac{\pi}{4}(kr_1-3)\right)
	 + O\left(x^{\frac{mk}{2}-\frac{3}{5}}\alpha^{\frac{1}{2}+\epsilon}\right).
\end{align*}
Substituting this into \eqref{identity after using cauchy residue}, we complete the proof of our lemma.
\end{proof}

 Before proceeding to the proof, we first establish the following lemma.
       \begin{lemma}
      	We have 
       	\begin{equation*}
       		\max_{-1\leq h\leq 1} \Delta_{\mathbb{K}}^{(k)}\left(x^{mk}e^h\right) \geq \frac{\alpha^{\frac{1}{mk}}}{\sqrt{\pi}} \int_{-\infty}^{\infty} \Delta_{\mathbb{K}}^{(k)}\left(x^{mk}e^{u/x}\right) \exp\left(-u^2\alpha^{\frac{2}{mk}}\right)\mathrm{d}u +O(x^{mk-1}e^{-x^2}).
       	\end{equation*}
       \end{lemma}
       \begin{proof}
       	We start with 
       	\begin{align*}
       		&\frac{\alpha^{\frac{1}{mk}}}{\sqrt{\pi}} \int_{|u|\leq x} \Delta_{\mathbb{K}}^{(k)}\left(x^{mk}e^{u/x}\right) \exp\left(-u^2\alpha^{\frac{2}{mk}}\right)\mathrm{d}u \\
       		\leq &\max_{-1\leq h\leq 1} \Delta_{\mathbb{K}}^{(k)}\left(x^{mk}e^h\right) \frac{\alpha^{\frac{1}{mk}}}{\sqrt{\pi}}\int_{-\infty}^{\infty} \exp\left(-u^2\alpha^{\frac{2}{mk}}\right) \mathrm{d}u \\
       		\leq & \max_{-1\leq h\leq 1} \Delta_{\mathbb{K}}^{(k)}\left(x^{mk}e^h\right).
       	\end{align*}
       Also, using the bound $\Delta_{\mathbb{K}}^{(k)}\left(y\right) = O(y)$, we find
       \begin{align*}
       &\frac{\alpha^{\frac{1}{mk}}}{\sqrt{\pi}}	\int_{x}^{\infty}\Delta_{\mathbb{K}}^{(k)}\left(x^{mk}e^{u/x}\right) \exp\left(-u^2\alpha^{\frac{2}{mk}}\right)\mathrm{d}u\\
       \ll \,& x^{mk} \alpha^{\frac{1}{mk}} \int_{x}^{\infty} \exp\left( \frac{u}{x}- u^2\alpha^{\frac{2}{mk}}\right)\mathrm{d}u\\
       =\, &x^{mk} \exp\left(\frac{1}{4x^2\alpha^{\frac{1}{mk}}}\right) \int_{x\alpha^{\frac{1}{mk}}} \exp\left( - \left(u-\frac{1}{2x\alpha^{\frac{1}{mk}}}\right)^2\right)\mathrm{d}u\\
      \ll & \, x^{mk-1} e^{-x^2}.
       \end{align*}
    A similar calculation gives 
   \begin{equation*}
   	\frac{\alpha^{\frac{1}{mk}}}{\sqrt{\pi}}	\int_{-\infty}^{-x}\Delta_{\mathbb{K}}^{(k)}\left(x^{mk}e^{u/x}\right) \exp\left(-u^2\alpha^{\frac{2}{mk}}\right)\mathrm{d}u \ll x^{mk-1} e^{-x^2}.
   \end{equation*}
Combining all the above estimates, we conclude the proof.
       \end{proof}

\section{Proof of Theorem \ref{Main Result} }
          \begin{proof}
  We start the proof by applying the above lemma in Proposition \ref{Prop} and truncate the infinite series in \eqref{Proposition equation}  at $n=X^{8/5}$. For the range $X/2\leq x\leq X^{8/5}$, we have
\begin{align}\label{Main inequality of proof}
	\max_{-1\leq h\leq 1} \Delta_{\mathbb{K}}^{(k)}\left(x^{mk}e^h\right) \geq &\frac{D^{\frac{1}{2m}}}{\pi \sqrt{mk}}x^{\frac{mk-1}{2}} \sum_{n=1}^{X^{8/5}} \frac{d_{\mathbb{K}}^{(k)}(n)}{n^{\frac{mk+1}{2km}}}\exp\left( -\pi^2\left(\frac{n}{D^k\alpha}\right)^{\frac{2}{mk}}\right)\\
	& \nonumber \ \ \ \ \ \ \ \ \ \ \ \ \ \ \ \ \ \ \ \ \ \ \ \ \times \cos\left(2\pi mk x \left(\frac{n}{D^k}\right)^{\frac{1}{mk}}+ \frac{\pi}{4}(kr_1-3)\right)
	+ O\left(x^{\frac{mk}{2}-\frac{3}{5}}\alpha^{\frac{1}{2}+\epsilon}\right).
\end{align} 
For  $0\leq \nu \leq m$,  $P_{\nu}$ denotes the set of all rational primes which are unramified in $\mathbb{K}$ and which split into exactly $\nu$ distinct $\mathcal{O}_{\mathbb{K}}$-prime ideals of degree $1$. Let $\delta_{\nu}$ be the Dirichlet density of $P_{\nu}$ as in the introduction and define $\mathcal{I}=\{1\leq \nu \leq m: \delta_{\nu}>0\}$. Then $R=|\mathcal{I}|$. We construct a resonator set $\mathcal{M}$ to be the set of all square free  $n\in [C_1^{mk}\alpha, 2^{mk}\alpha]$ such that $n$ has $[\mu_{\nu}\log_2 \alpha]$ distinct prime factors from $P_{\nu}$ for all $\nu \in I$ and no other, where $C_1$ is a constant satisfying $0<C_1<2$ and $\mu_{\nu} > 0$ is a positive constant to be determined later. Let $M$ be the cardinality of $\mathcal{M}$.

We notice that the value of $d_{\mathbb{K}}^{(k)}(n)$ is very large when $n\in \mathcal{M}$. Let $p\in P_{\nu}$ be  prime. Then the factorization of $\mathcal{O}_{\mathbb{K}}$-prime ideal contains exactly $\nu$ prime ideals $\mathfrak{p}$ such that $N(\mathfrak{p})=p$. So that $d_{\mathbb{K}}^{(1)}(p) = \nu$. Also, for every prime ideal $\mathfrak{p}$, consider $k$ tuples $(\mathfrak{n}_1, \mathfrak{n}_2, \cdots, \mathfrak{n}_k)$ such that $N(\mathfrak{n}_1 \mathfrak{n}_2 \cdots \mathfrak{n}_k)=p$ by choosing one of the $\mathfrak{n}_i=\mathfrak{p}$ and rest to be $\mathcal{O}_{\mathbb{K}}$. We see that $d_{\mathbb{K}}^{(k)}(p)=\nu k$. Similarly, one may get $d_{\mathbb{K}}^{(k)}(p^r)\geq\nu k$ for all                                                                                                                                                                       $r\geq 1$. We obtain that for $n\in \mathcal{M}$,
\begin{equation} \label{lower bound of d_K}
	d_{\mathbb{K}}^{(k)}(n) \geq \prod_{\nu \in \mathcal{I}} (k\nu)^{[\mu_{\nu} \log_2 \alpha]}.
\end{equation}
We now proceed to apply Theorem~\ref{K-M Theorem}. For this purpose, we consider
\begin{equation*}
	f(n) = 
	\begin{cases}
		\displaystyle d_{\mathbb{K}}^{(k)}(n)n^{-\frac{mk+1}{2mk}} \exp\left( -\pi^2 \left(\left(\frac{n}{D^k\alpha}\right)^{\frac{2}{mk}}\right)\right), & \text{if } n \leq X^{8/5}, \\[1ex]
		0, & \text{otherwise},
	\end{cases}
\end{equation*}
$\lambda_n = 2\pi mk  (n/D^k)^{1/(mk)}$ and $\theta= (kr_1-3)\pi/4$. Also, we choose $A_1=8/5, A_2=3/2, A_3=1$ and $A_4=9/10$. Then using Theorem \ref{K-M Theorem}, there exists $x$ satisfying  $\sqrt{X}\leq x\leq 5X^{3/2}\log^2 X$ such that 
\begin{equation} \label{Lower bound of A(x)}
 \left|\sum_{n\leq X^{8/5}}f(n)\cos(\lambda_nx +\theta)\right|	 \geq \cfrac{\pi}{4e} \sum_{n\in \mathcal{M}}f(n) + O\left( X^{-1/2}e^{2M/C_1}\left(\sum_{\lambda_n\leq 4\alpha}a_n\right)\right) +  O\left(\frac{X^{-9/10}}{\alpha}\sum_{n\leq X^{8/5}}a_n\right).
\end{equation}
In order to estimate the lower bound of $\sum_{n\in \mathcal{M}} f(n)$, it is necessary to understand the size of $\mathcal{M}$. From the proof of Theorem 1 in \cite{GKMN}, one find
\begin{equation*}
	M\asymp \frac{\alpha}{\log \alpha} \prod_{\nu \in \mathcal{I}}\frac{\Big( \delta_{\nu} \log_2 \alpha\Big)^{[\mu_{\nu} \log_2 \alpha]}}{[\mu_{\nu} \log_2 \alpha]!} \asymp \alpha (\log \alpha)^{\kappa} (\log_2 \alpha)^{-\mu},
\end{equation*}
where $\kappa = -1 +\sum_{\nu \in \mathcal{I}} \mu_{\nu}(1+\log\delta_{\nu}-\log \mu_{\nu})$ and $\mu = R/2$. 
 We choose  the parameter $\alpha$ such that $M$ is of order $\log X$.
In fact,
\begin{equation*}
	\alpha = C(k) (\log X) (\log_2 X)^{-\kappa} (\log_3 X)^{\mu},
\end{equation*}
where $C(k)$ is a small positive constant. 

Note that $\log \alpha \asymp \log_2 X$ and $\log_2 \alpha \asymp \log_3 X$. Also, for $n\in \mathcal{M}$, we have $(n/D^k \alpha)^{2/(mk)}\ll 1$. Using these and the lower bound in \eqref{lower bound of d_K}, we find that 
\begin{align} \label{lower bound of sum_ f(n)}
\nonumber	\sum_{n\in \mathcal{M}} f(n) &\geq 2^{\frac{1}{2}(1-mk)} \alpha^{-\frac{mk+1}{2mk}} \prod_{\nu \in \mathcal{I}} (k\nu)^{[\mu_{\nu} \log_2 \alpha]}\sum_{n\in \mathcal{M}}1\\
\nonumber & \gg \alpha^{-\frac{mk+1}{2mk}} (\log \alpha)^{\sum_{\nu \in \mathcal{I}} \mu_{\nu} \log(k\nu)}M\\
\nonumber	& \gg \alpha^{\frac{mk-1}{2mk}} (\log \alpha)^{\kappa+ \sum_{\nu \in \mathcal{I}}\mu_{\nu} \log (k\nu )} (\log_2 \alpha)^{-\mu}\\
\nonumber	& \gg (\log X)^{\frac{mk-1}{2mk}} (\log_2 X)^{\frac{mk+1}{2mk}\kappa + \sum_{\nu \in \mathcal{I}} \mu_{\nu} \log (k\nu) } (\log_3 X)^{-\frac{mk+1}{2mk}\mu}\\
	& \gg (\log X)^{\frac{mk-1}{2mk}} (\log_2 X)^{\beta} (\log_3 X)^{\gamma},
\end{align}
where $\beta$ and $\gamma$ are defined as in  Theorem \ref{Main Result}. In the last step we optimize $\mu_{\nu}$ in the power of $\log_2 X$ and choose $\mu_{\nu}= \delta_{\nu} (m\nu)^{2km/(km+1)}$.
Our aim is  to show that the error terms on the right hand side of \eqref{Lower bound of A(x)} are smaller than the main term. Observe that $M\asymp C(k)\log X$. Again we choose $C(k)$ to be very small such that 
\begin{equation*}
	e^{2M/C_1}\ll X^{1/4-\epsilon}.
\end{equation*}
Now let us compute the following sums:
\begin{align} \label{For error 1}
\nonumber	\sum_{\lambda_n\leq 4\alpha}f(n) &\ll \sum_{n\leq 4\alpha} d_{\mathbb{K}}^{(k)}(n) n^{-\frac{mk+1}{2mk}}\\
\nonumber	&= \int_{1}^{4\alpha} t^{-\frac{mk+1}{2mk}} \mathrm{d}\left(\sum_{n\leq t}d_{\mathbb{K}}^{(k)}(n)\right)\\
	&\ll \alpha^{\frac{mk-1}{2mk}} \log(4 \alpha)^{k-1} \ll X^{\epsilon},
\end{align}
where we used the bound $\sum_{n\leq t} d_{\mathbb{K}}^{(k)}(n) \ll t(\log t)^{k-1}$. Similarly we find that
\begin{align} \label{for error 2}
	\sum_{n\leq X^{8/5}} f(n) \ll X^{\frac{4(mk-1)}{5mk}} (\log X)^{k-1} \ll X^{\frac{4}{5}- \frac{1}{5mk} +\epsilon}.
\end{align}
The error can now be evaluated using the two estimates obtained above.
\begin{align*}
	\text{Error} &\ll X^{-1/2} e^{2M/C_1} \left(\sum_{\lambda_n\leq 4\alpha}f(n)\right)+ \frac{X^{-9/10}}{\alpha} \left( \sum_{n\leq X^{8/5}}f(n)\right)\\
	&\ll X^{-1/4+\epsilon} + X^{-\frac{1}{10}-\frac{4}{5mk}+\epsilon} \ll X^{-\frac{1}{10} +\epsilon}.
\end{align*}
Employing \eqref{lower bound of sum_ f(n)} in \eqref{Lower bound of A(x)} we conclude that for a large $X$,
\begin{equation*}
\max_{\frac{X}{2}\leq x \leq 5 X^{3/2}(\log X)^2} \left|\sum_{n\leq X^{8/5}}f(n)\cos(\lambda_nx +\theta)\right|\gg (\log X)^{\frac{mk-1}{2mk}} (\log_2 X)^{\beta} (\log_3 X)^{\gamma}.
\end{equation*}

Hence, the inequality \eqref{Main inequality of proof} implies 
\begin{equation*}
	\max_{-1\leq h\leq 1} \Delta_{\mathbb{K}}^{(k)}\left(x^{mk}e^h\right) \gg x^{\frac{mk-1}{2}} (\log x)^{\frac{mk-1}{2mk}} (\log_2 x)^{\beta} (\log_3 x)^{\gamma}.
\end{equation*}
This completes the proof of our theorem.
 \end{proof}

\section{Some Special Examples}
In this section, we give some examples to explain Theorem \ref{Main Result}. In Examples~\ref{Example 2} and \ref{Example 3}, we look at normal and non-normal extensions, and these examples improve the ones given in \cite{GKMN}. We give an example of non-normal extension of degree $5$ in Example~\ref{Example 4}. Similar computations for higher degree extensions can also be worked out (see Section~4 of \cite{GKMN}).
\subsection{} 
Setting $\mathbb{K} = \mathbb{Q}$ and choosing any $k \geq 2$, Theorem \ref{Main Result} reproduces the lower bound for the Piltz divisor problem previously proved by the second author \cite{K. Mahatab}.
%

\subsection{} \label{Example 2} We consider $\mathbb{K}$ to be a normal extension of $\mathbb{Q}$ with degree $m\geq2$. Then the sets $P_{\nu}$ is empty for $1\leq \nu <m$. In this case, the Dirichlet densities  $\delta_{\nu}=0$ for $1\leq \nu <m$ and $\delta_{m}=1/m$. Thus, for any integer $k\geq 1$, Theorem \ref{Main Result} gives 
\begin{equation*}
	\Delta_{\mathbb{K}}^{(k)}(x) =\Omega\left((x\log x)^{\frac{mk-1}{2mk}} (\log_2x)^{\frac{1}{2}k(mk+1)(mk)^{-2/(mk+1)}-\frac{1}{2}-\frac{1}{2mk}} (\log_3 x)^{-\frac{1}{4}-\frac{1}{4mk}}\right).
\end{equation*}
This result improves upon the bounds previously obtained by Hafner \cite{Halfner(2)} and by Girstmair et al. \cite{GKMN}.
\subsection{} \label{Example 3} We take $\mathbb{K}$ to be a non-normal extension of $\mathbb{Q}$ with $[\mathbb{K}:\mathbb{Q}]=3$. Let us choose a field $L$ such that $L=\mathbb{K}(\sqrt{D})$, where the discriminant  $D$ is not a perfect square. Then $\mathbb{K}^G$ is the minimal normal extension of $\mathbb{Q}$ and its Galois group $G$ is isomorphic to $S_3$, the symmetric group of three elements. Also, $L$ is a quadratic extension of $\mathbb{K}$. So, the Galois group $H=\text{Gal}(L/K)$ is isomorphic to the cyclic group $\mathbb{Z}_2$. Then the Dirichlet densities are given by $\delta_1 =1/2$, $\delta_2=0$ and $\delta_3=1/6$.
Then for any $k\geq 1$, 
\begin{equation*}
	\Delta_{\mathbb{K}}^{(k)} = \Omega\left( (x\log x)^{\frac{1}{2}- \frac{1}{6k}}(\log_2 x)^{\beta}(\log_3 x)^{-\frac{1}{2}-\frac{1}{6k}}\right),
\end{equation*}
where
\begin{equation*}
	\beta = \frac{(3k+1)}{12}k^{(3k-1)/(3k+1)}\left(3^{\frac{3k-1}{3k+1}}+1\right)-\frac{1}{2}-\frac{1}{6k}.
\end{equation*}

\subsection{} \label{Example 4} Let $f(T)= T^5+20T+16$. Then $f(T)$ is an irreducible polynomial over $\mathbb{Q}$. We consider $\mathbb{K}=\mathbb{Q}(\xi)$, where $\xi$ is a root of $f$ and $m=[\mathbb{K}: \mathbb{Q}]=5$. It is clear that $K$ is not normal extension. We denote $\mathbb{K}^{G}$ as the Galois closure of the extension $\mathbb{K}/\mathbb{Q}$. Then $G= \text{Gal}(\mathbb{K}^{G}/\mathbb{Q}) \cong A_5$ and $H= \text{Gal}(\mathbb{K}^{G}/\mathbb{K})\cong A_4$,  where $A_n$ the alternating group on $n$ elements.
There are exactly three non-empty sets $P_{\nu} \ (\nu =1,2,5)$, which are in $\mathbb{K}$. The corresponding non-zero Dirichlet densities are $\delta_1 = 1/4$, $\delta_2 = 1/3$ and $\delta_5 =1/60$. Thus, for each $k\geq 2$, 
\begin{equation*}
	\Delta_{\mathbb{K}}^{k}(x) = \Omega\left( (x\log x)^{\frac{1}{2}-\frac{1}{10k}} (\log_2 x)^{\beta} (\log_3 x)^{-\frac{3}{4}- \frac{3}{20k}}\right),
\end{equation*}
where 
\begin{equation*}
	\beta = \frac{(5k+1)}{600}k^{(5k-1)/(5k+1)}\left( 15+ 5\cdot 2^{\frac{15k+1}{5k+1}}+ 5^{\frac{10k}{5k+1}}\right) -\frac{1}{2} - \frac{1}{10k}.
\end{equation*}



	\section*{Acknowledgment}  
N. Karak gratefully acknowledges the financial support from the Prime Minister's Research Fellowship, Government of India (PMRF ID: 2403449). K. Mahatab is supported by the DST INSPIRE
Faculty Award Program and grant no. DST/INSPIRE/04/2019/002586.

	\end{document}